\documentclass[10pt,a4paper]{paper}

\usepackage{amsmath}
\usepackage{amssymb}
\usepackage{amsthm}		% newtheorem
\usepackage{pifont} 	% ding, cmark, xmark
\usepackage[dvipsnames, table]{xcolor} % color, red
\usepackage[textwidth=80]{todonotes}  % Notes : todo, note
\usepackage{soul}		% st -- strikeout

\usepackage[mathletters]{ucs}	% support of math unicode characters in non-unicode environment
\usepackage[utf8x]{inputenc}	% support of unicode characters 	
\usepackage{ytableau} 	% ytableaushort
\usepackage{listings}	% lstlisting
\usepackage{enumitem} 	% definition todolist
\usepackage{pifont}		% definition todolist
\usepackage{algpseudocode} 	% algorithms
\usepackage{algorithm}	% algorithm border
\usepackage{graphicx}	% rotation \rotatebox
\usepackage{array}   	% for \newcolumntype macro
\usepackage{xspace}		% for \xspace
\usepackage{graphicx}	% for \includegraphics
\usepackage{cancel}		% for \cancel
\usepackage{young}		% for Young-environment
\usepackage{tikz}
\RequirePackage[pdftex,
colorlinks=true,
linkcolor=black,
citecolor=black,
urlcolor=black,
filecolor=black]{hyperref} % hid boexs around links
\usepackage{cleveref} 	% cref, Cref
\usepackage{emptypage}	% no header and footer on empty page 
\usepackage{mathtools}	% \coloneqq
\usepackage{fancyhdr} 	% modification of header
\usepackage[most]{tcolorbox} %restate theorems
\usepackage{xr}
\usetikzlibrary{graphs}
\RequirePackage[left=3cm,right=3cm,top=3.5cm,bottom=3cm]{geometry}
\usepackage{wrapfig}
\usetikzlibrary{calc, matrix}
\usepackage{eso-pic}

\usetikzlibrary{decorations.pathreplacing} %tikz tableau
\usetikzlibrary{arrows}
\usetikzlibrary{arrows.meta}
\newcommand{\tikznode}[3][inner sep=0pt]{\tikz[remember
	picture,baseline=(#2.base)]{\node(#2)[#1]{\scriptsize{$#3$}};}}

\newcolumntype{L}{>{$}l<{$}} % math-mode version of "l" column type
\definecolor{codegreen}{rgb}{0,0.6,0}
\definecolor{codegray}{rgb}{0.5,0.5,0.5}
\definecolor{codepurple}{rgb}{0.58,0,0.82}
\definecolor{backcolour}{rgb}{0.95,0.95,0.95}

\lstdefinestyle{mystyle}{
	backgroundcolor=\color{backcolour},   
	commentstyle=\color{codegreen},
	firstnumber=auto,
	keywordstyle=\color{magenta},
	numberstyle=\tiny\color{codegray},
	stringstyle=\color{codepurple},
	basicstyle=\fontsize{9}{10}\ttfamily\color{Blue},
	breakatwhitespace=false,         
	breaklines=true,                 
	captionpos=b,                    
	keepspaces=true,                 
	numbers=none,                    %none
	numbersep=5pt,                  
	showspaces=false,                
	showstringspaces=false,
	showtabs=false,                  
	tabsize=2
}
\lstdefinestyle{mystyle_small}{
    basicstyle=\fontsize{7}{9}\selectfont\ttfamily\color{Blue},
	backgroundcolor=\color{backcolour},   
	commentstyle=\color{codegreen},
	firstnumber=auto,
	keywordstyle=\color{magenta},
	numberstyle=\tiny\color{codegray},
	stringstyle=\color{codepurple},
	breakatwhitespace=false,         
	breaklines=true,                 
	captionpos=b,                    
	keepspaces=true,                 
	numbers=none,                    %none
	numbersep=5pt,                  
	showspaces=false,                
	showstringspaces=false,
	showtabs=false,                  
	tabsize=2
}
\swapnumbers
\numberwithin{equation}{section}

\newtheorem{theorem*}[equation]{Theorem*}
\newtheorem{theorem}[equation]{Theorem}
\newtheorem{lemma}[equation]{Lemma}

\newtheorem{definition}[equation]{Definition}
\newtheorem{proposition}[equation]{Proposition}
\newtheorem{claim}[equation]{Claim}

\newcommand{\rank}{\ensuremath{\mathsf{rank}}\xspace}

\newcommand{\abs}[1]{\ensuremath{\left|#1\right|}}

\makeatletter
\newcommand*{\rom}[1]{\expandafter\@slowromancap\romannumeral #1@}
\makeatother

\newcommand{\mc}[1]{\ensuremath{\mathcal{#1}}}

\newcommand{\NN}{\ensuremath{\mathbb{N}}}
\newcommand{\CC}{\ensuremath{\mathbb{C}}}
\newcommand{\QQ}{\ensuremath{\mathbb{Q}}}

\newcommand{\sym}{\ensuremath{\mathfrak{S}}}

\newcommand{\weyl}{\ensuremath{\mathbf{W}}\xspace}
\newcommand{\schur}{\ensuremath{\mathbf{S}}\xspace}
\newcommand{\preschur}{\ensuremath{\mathbf{X}}\xspace}

\newcommand{\Hom}{\ensuremath{\mathsf{Hom}}\xspace}

\newcommand{\la}{\lambda}

\newcommand{\ghs}{\ensuremath{\mathsf{HS}}}

\let\oldytableaushort\ytableaushort

\renewcommand{\ytableaushort}[1]{ {\ytableausetup{centertableaux,notabloids} \oldytableaushort{#1}} }
\ytableausetup{mathmode, boxsize=1em,centertableaux ,smalltableaux}

\newcommand{\GL}{\ensuremath{\mathsf{GL}}}

\newcommand{\pierimaps}{{\texttt{PieriMaps}}\xspace}

\newlist{todolist}{itemize}{2}
\setlist[todolist]{label=$\square$}

\title{Young Flattenings in the Schur module basis}
\author{Lennart J. Haas and Christian Ikenmeyer}

\begin{document}

\maketitle

\noindent April 2021

\begin{abstract}
There are several isomorphic constructions for the irreducible polynomial representations of the general linear group in characteristic zero. The two most well-known versions are called Schur modules and Weyl modules. Steven Sam used a Weyl module implementation in 2009 for his Macaulay2 package PieriMaps. This implementation can be used to compute so-called Young flattenings of polynomials.
Over the Schur module basis Oeding and Farnsworth describe a simple combinatorial procedure that is supposed to give the Young flattening, but their construction is not equivariant.
In this paper we clarify this issue, present the full details of the theory of Young flattenings in the Schur module basis, and give a software implementation in this basis. Using Reuven Hodges' recently discovered Young tableau straightening algorithm in the Schur module basis as a subroutine, our implementation outperforms Sam's PieriMaps implementation by several orders of magnitude on many examples, in particular for powers of linear forms, which is the case of highest interest for proving border Waring rank lower bounds.
\end{abstract}


\smallskip


\smallskip

\noindent{\footnotesize \textbf{ACM Subject Classification:} Mathematics of computing $\to$ Mathematical software \\
Computing methodologies $\to$ Symbolic and algebraic manipulation $\to$ Computer algebra systems\\
Theory of computation $\to$ Computational complexity and cryptography $\to$ Algebraic complexity theory
}

\section{Motivation}
Young flattenings of polynomials are equivariant linear maps from a space of homogeneous polynomials to a space of matrices, where the row and column space are irreducible representations of $\GL_n := \GL(\mathbb C^n)$. One is usually interested in finding lower bounds for the rank of the image of a Young flattening, as it can be used to obtain lower bounds on the border Waring rank of a polynomial, and more generally for any border $X$-rank for a $\GL_n$-variety $X$, i.e., 
given a point $p$ to find a lower bound on the smallest $i$ such that $p$ lies on the $i$-th secant variety of $X$, see e.g.\ \cite{Landsberg2015}.
One early example are Sylvester's \emph{catalecticants} \cite{sylvester:1970:principles}.
Landsberg and Ottaviani \cite{landsberg_ottaviani:2011:new_lower_bounds_for_the_border_rank_of_matrix_multiplication} use Young flattenings in the tensor setting. The name \emph{Young flattening} was introduced in the predecessor paper \cite{landsbergottaviani2013}.
Young flattenings also appear in disguise in the area of algebraic complexity theory as \emph{matrices of partial derivatives}, \emph{shifted partial derivatives}, \emph{evaluation dimension}, and \emph{coefficient dimension} \cite{nisan_wigderson:1995:lower_bounds_on_arithmetic_circuits, gupta_et_al:2014:approaching_a_chasm_at_depth_four}.
They can in principle be used to find computational complexity lower bounds in many algebraic computational models such as border determinantal complexity (see \cite{landsbergmanivelressayre2013}) and border continuantal complexity \cite{bringmann:2018:on_algebraic_branching_programs_of_small_width}, which makes Young flattenings an interesting tool in the Geometric Complexity Theory approach by Mulmuley and Sohoni \cite{mulmuley_sohoni:2002:GCT}, \cite{mulmuley_sohoni:2008:GCTII}, \cite{burgisserlandsbergmanivelweyman2011}.
Limits of these methods (in the case of studying $X$-rank) have recently been proved in \cite{efremenko_garg_oliveira_wigderson:2017:BFRM, garg_makam_olveira_wigderson:2019:MBfRM}.
No such limits are known for using Young flattenings to study the orbit closure containment problems in geometric complexity theory. First results in this direction were obtained in \cite{efremenko_landsberg_schenck_weyman:2018:shifted_partial_derivatives}, where limits to the method of shifted partial derivatives are shown.
This was improved on in \cite{gesmudo_landsberg:2019:EPS_MS_PD}, where a setting was given in which Young flattenings give strictly more separation information than partial derivatives.

The \emph{Waring rank} of a homogeneous degree $d$ polynomial $p \in \schur^dV$ is defined as the smallest $r$ such that $p$ can be written as a sum of $r$ many $d$-th powers of homogeneous linear forms (arbitrary linear combinations of $d$-th powers are usually allowed if the base field is not algebraically closed).
For example $(x-y)^3+y^3 = x^3 - 3x^2y + 3xy^2$, hence $x^3 - 3x^2y + 3xy^2$ has Waring rank at most 2.
The \emph{border Waring rank} of $p$ is the smallest $r$ such that $p$ can be approximated arbitrarily closely coefficient-wise by polynomials of Waring rank at most $r$.
For example $3 \varepsilon x^2 y = \lim_{\varepsilon\to 0} ((x+\varepsilon y)^3-x^3)$, hence $x^2y$ has border Waring rank at most 2.

If a Young diagram $\la$ is contained in another Young diagram $\mu$ such that the column lengths of both diagrams differ by at most 1 in each column, then we have 
a unique nonzero equivariant map between
$\schur^dV \otimes \schur^\la V \rightarrow \schur^\mu V,
$
where $\schur^\la V$, $\schur^\mu V$, and $\schur^d V$ are irreducible polnomial $\GL_n$-representations, and $d$ is the difference in the number of boxes of $\mu$ and $\lambda$.
This is called the \emph{Pieri map}, and it induces a linear map
$\mc{F}_{λ,μ} : \schur^dV \to \text{End}(\schur^λV,\schur^μV).$
Since border Waring rank is subadditive, a lower bound on the border Waring rank of $p$ is obtained by rounding up to quotient of ranks 
\begin{equation}\label{eq:rankquotient}
\left\lceil\frac{\textup{rank}(\mc{F}_{λ,μ}(p))}{\textup{rank}(\mc{F}_{λ,μ}(x^d))}\right\rceil,\end{equation}
where $x$ is some variable that appears in $p$, and rank$(.)$ is the usual rank of matrices.

There are several isomorphic constructions for the irreducible polynomial representations of the general linear group in characteristic zero. The two best known versions are called Schur modules and Weyl modules and they only differ in the order of the row-symmetrizer and the column-symmetrizer in their definition of the Young symmetrizer.
This results in different bases for the irreducible representations.
Sometimes results that are proved in one basis are reproved in the other basis, but the proofs look significantly different (see e.g.\ \cite{bci:10} and \cite{manivelmichalek2014}). In fact, so far some results are only provable in a natural way over one basis and not the other, see e.g.\ \cite{ressayre2020}.
Based on an explicit paper by Olver over the Weyl module basis \cite{olver:1982:differential_hyperform}
Steven Sam in 2009 implemented his Macaulay2 package PieriMaps \cite{sam:2008:computing_inclusions_of_schur_modules:macaulay2},
which among other things can be used to compute the rank quotient \eqref{eq:rankquotient}, see Section~\ref{subsec:example_pieri} below.

The papers \cite{farnsworth:2015:koszul-young_flattenings} (in its Section 5\footnote{Although the description in the paper is wrong, the use of the software package is correct and gives the result claimed in the paper.}) and \cite{oeding:2016:border_ranks_of_monomials_v1} (only in version 1) describe the PieriMaps package as if it would be working in the Schur module basis and they assume that the Young flattenings have an extremly simple combinatorial description.
However, this is wrong (see Section~\ref{sec:oversimplification} below), which led to a revision of \cite{oeding:2016:border_ranks_of_monomials_v1}.

In this paper we work out the details of Young flattenings in the Schur module basis: We closely mimic the arguments in \cite{olver:1982:differential_hyperform}, but we take care of subtle sign issues that are not present in Olver's work over the Weyl module basis.
We then make use of a recent fast algorithm (and implementation) by Reuven Hodges for Young tableau straightening in the Schur module basis \cite{hodges:2017:CNIF}
to get a highly efficient Young flattening algorithm that outperforms Sam's PieriMaps implementation by several orders of magnitude in many examples. We obtain the most impressive speedup factor of 1000 for flattening the power of a linear form, which is the denominator of \eqref{eq:rankquotient}.

Our contribution is therefore twofold: We thoroughly clarify the theory of Young flattenings in the Schur module basis
and we present a new and efficient implementation for Young flattenings that uses Hodges' state-of-the-art straightening algorithm over the Schur module basis.

\section{Preliminaries}
A \emph{composition} $\nu$ of a number $d$ is a finite list of natural numbers adding up to $d$, i.e., $(3,0,2,4)$ is a composition of 9.
A \emph{partition} is a nonincreasing composition, for example $λ=(6,4,3)$ is a partition.
We write $\la\vdash d$ if $\la$ is a partition of $d$.
We write $λ_i=0$ if $i$ is greater than the number of entries in $λ$.
We define $\ell(\la) := \min\{i \mid \la_i =0\}-1$.
We identify a partition with its \emph{Young diagram}, which is a top-left justified array of boxes, i.e., the set of points $\{(i,j) \mid 1 \leq i \text{ and } 1 \leq j \leq λ_i\}$. For example, the Young diagram corresponding to $(6,4,3)$ is 
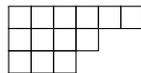
\[
\ytableaushort{\ \ \ \ \ \ ,\ \ \ \ ,\ \ \ }
\]
and we have $(2,4) \in λ$ and $(4,2) \notin λ$.
We see that $\ell(\la)$ is the number of rows of the Young diagram corresponding to $\la$.
We denote by $|λ|$ the number of boxes in $λ$, i.e., $|λ|=\sum_i λ_i$.
We denote by $\la^*$ the Young diagram obtained by reflecting $\lambda$ at the main diagonal, e.g., $(6,4,3)^* = (3,3,3,2,1,1)$. It follows that $\lambda_i^*$ is the length of the $i$-th column of $\lambda$.
We write $\lambda \subseteq \mu$ if for all $(i,j) \in \lambda$ we have $(i,j) \in \mu$. If $\lambda \subseteq \mu$, then
we denote by $\mu/\lambda$ the set of points that are in $\mu$ but not in $\lambda$. We call $\mu/\lambda$ a \emph{horizontal strip} if it has at most 1 box in each column. In this situation we write $\mu/\lambda \in \ghs$.

A Young diagram $λ$ whose entries are labeled with numbers is called a \emph{Young tableau} of shape~$λ$. For example,
\[
\ytableaushort{634233,3312,663}
\]
is a Young tableau of shape $(6,4,3)$.
A Young tableau is called \emph{semistandard} if the entries strictly increase in each column from top to bottom and do not decrease in each row from left to right. For example, \ytableaushort{1113,23} is a semistandard tableau.
We denote by $\sym_{λ}$ the symmetric group on the set $\{(i,j) \mid (i,j) \in \lambda\}$ of positions in $\lambda$. The group $\sym_{λ}$ acts on the set of all Young tableaux of shape $λ$ by permuting the positions. We write $\sigma T$ for the permuted Young tableau, where $λ$ is the shape of $T$ and $\sigma \in \sym_{λ}$.
For a subset $S \subseteq λ$ of positions we write $\sym_S$ to denote the symmetric group that permutes only the positions in $S$ among each other and fixes all other positions.

	Let $V^{\otimes λ} := V^{\otimes |λ|}$ be the $|\la|$-th tensor power of a vector space $V$ and associate to every tensor factor $V$ a position in $λ$.
	A rank 1 tensor $v = v_1 \otimes v_2 \otimes \cdots \otimes v_{|\la|}$ can now be represented by a Young diagram in whose $i$-th box we write the vector $v_i$.
	If we fix a basis $v_1, \dots, v_n$ of $V$, then a basis of $V^{\otimes λ}$ is obtained by all ways of writing $\{v_1,\ldots,v_n\}$ into the boxes of $λ$, allowing repetitions. If the fixed basis is clear from the context, then we write $i$ instead of $v_i$ into the boxes and obtain a Young tableau. The basis vector corresponding to the Young tableau $T$ is also denoted by $T$ when no confusion can arise, so for example if $x=v_1+v_2$ we can use the multilinearity of the tensor product to write
	\[
    \ytableaushort{x{v_2},{v_1}} = \ytableaushort{12,1} + \ytableaushort{22,1}
	\]

\subsection{The Weyl module basis}

Let λ be a Young diagram and let $(i, j) \in λ$ such that $(i+1, j) \in λ$ (i.e., the box below $(i,j)$ is still in $λ$). Then we define 	$B_{i,j} \coloneqq \{(i, k) \mid j \le k \le λ_i\} \cup \{(i+1, k)\mid 1 \le k \le j \}.$
	Pictorially, $B_{i,j}(λ)$ is the subset of the boxes of $λ$ given by collecting all boxes on the following path: Start at position $(i+1, 1)$ and move from left to right along row $i+1$ to box $(i+1,j)$, then switch the row to $(i,j)$ and move along row $i$ until reaching  $(i,λ_i)$. For example, $B_{1,3}((6,4,3))$ is given by the dotted boxes in the following diagram:
\[
\ytableaushort{\ \ \bullet \bullet\bullet\bullet,\bullet\bullet\bullet\ ,\ \ \ }.
\]

	\begin{definition}[Weyl module, {\cite[2.1.15]{weyman:2003:cohomology_of_vector_bundles_and_syzygies}}]
		\label{def:weyl_module}
		
		The \emph{Weyl module} $\weyl^λ(V)$ is defined as the quotient space $V^{\otimes λ}/P^λ$, where
		$P^λ \subseteq V^{\otimes λ}$ is the linear subspace generated by the following two types of vectors:
		
		\begin{enumerate}
			\item (Symmetric relation) $T - σ T$, if σ is a permutation that preserves the row indices of all positions of $λ$ (in other words, $σ$ permutes \emph{within the rows} of $λ$).
			\item (Shuffle relation) $\sum_{σ \in \sym_{B_{i,j}(λ)}} σ T,$ if $i,j \in \NN$ such that $(i, j), (i + 1, j) \in λ$.
		\end{enumerate}
	\end{definition}
The Weyl modules for Young diagrams $\la$ with at most $\dim V$ rows
form a complete list of pairwise non-isomorphic irreducible polynomial representations of $\GL(V)$.
In this paper we will not work with Weyl modules, but with the isomorphic Schur modules, which are defined in the following section.

\subsection{The Schur module basis}
\label{subsec:schur_basis}
	Let $T$ be a Young tableau of shape $λ$. Let $1 \leq i<j \leq λ_1$ be two column indices. Let $B$ and $C$ be two equally large sets of boxes, $B$ from column $i$ and $C$ from column $j$.
	An \emph{exchange tableau of $T$ corresponding to $B$ and $C$} is defined as the tableau arising from $T$ by exchanging the content of the boxes $B$ with the content of the boxes $C$ while preserving the vertical order of the entries in $B$ and~$C$.
	We denote this exchange tableau by $E^{B}_{C}(T)$.
	For a subset $C$ of boxes from column $j$,
	we write $E^{i}_{C}(T) := \bigcup_{B} E^{B}_{C}(T)$, where $B$ ranges over all cardinality $|C|$ subsets of boxes in column $i$.
	\begin{definition}[Schur module]
		\label{def:schur_module}
		The \emph{Schur module} $\schur^λV$ is defined as the quotient space $V^{\otimes λ}/Q^λ$ where
		$Q^λ \subseteq V^{\otimes λ}$ is the linear subspace generated by the following vectors
		
		\begin{enumerate}
			\item (Grassmann relation) $T + T'$, where $T'$ is obtained from $T$ by swapping two elements in the same column.
			\item (Plücker relation) $T - \sum_{T'\in E^i_{C}(T)}$ for any $i$ and any subset $C$ of a column $j\neq i$ with $\la_i^*\geq |C|$.
		\end{enumerate}
	\end{definition}

For example, in $\schur^λV$ we have $\ytableaushort{12,14} = 0$ and we have
$\ytableaushort{12,34}= - \ytableaushort{32,14}$ by the Grassmann relation. The Plücker relation gives
$$\ytableaushort{12,34} =  \ytableaushort{21,34} + \ytableaushort{13,24}$$ via $E^1_{\{(1,2)\}}$. Note that $\{(1,2)\}$ is the set containing the single box in row 1 and column 2.

\subsection{The Schur module via moving boxes between columns}

If we only consider the Grassmann relation, then we call the quotient $\preschur^{\la^*} V$. More formally, let $G^\la \subseteq V^{\otimes \la}$ be the linear subspace spanned by the $T + T'$, where $T'$ is obtained from $T$ by swapping two elements in the same column. The quotient $V^{\otimes \la}/G^\la$ is denoted by $\preschur^{\la^*} V$. Clearly, in the language of skew-symmetric powers we have
\begin{equation}\label{eq:Xnuwedges}
\preschur^\nu V \simeq ({\textstyle\bigwedge}^{\nu_1} V) \otimes \dots \otimes ({\textstyle\bigwedge}^{\nu_{\ell(\nu)}} V).
\end{equation}
In terms of explicit basees, this isomorphism 
maps each column from top to bottom to a skew-symmetric tensor and vice versa: For example, $\ytableaushort{12,34}$ is mapped to $x_1 \wedge x_3 \otimes x_2 \wedge x_4$.
We define
\[
Y_d := \bigoplus_{\text{composition $\alpha$ of $d$}} \preschur^\alpha V
\]
as an \emph{outer} direct sum. Note that
each
$\preschur^\alpha V \subseteq V^{\otimes |\alpha|}$, but there is no obvious embedding of $Y_d$ in $V^{\otimes |\alpha|}$.
In fact, there is a natural isomorphism
\begin{equation}\label{eq:naturalisomgrassmannian}
Y_d \simeq {\textstyle\bigwedge}^d(V^{\oplus d})
\end{equation}
that can be described explicitly in terms of basis vectors using first the isomorphism \eqref{eq:Xnuwedges}:
A standard basis vector
\[
x_{k_{1,1}} \wedge \cdots \wedge x_{k_{\nu_1,1}}\otimes x_{k_{1,2}} \wedge \cdots \wedge x_{k_{\nu_2,2}} \otimes \cdots \otimes x_{k_{1,d}} \wedge \cdots \wedge x_{k_{\nu_d,d}}
\]
is mapped to
\begin{equation}\label{eq:longwedge}
x_{k_{1,1},1} \wedge \cdots \wedge x_{k_{\nu_1,1},1}\wedge x_{k_{1,2},2} \wedge \cdots \wedge x_{k_{\nu_2,2},2} \wedge \cdots \wedge x_{k_{1,d},d} \wedge \cdots \wedge x_{k_{\nu_d,d},d},
\end{equation}
for example $x_1 \wedge x_3 \otimes x_2 \wedge x_4$ is mapped to $x_{1,1} \wedge x_{3,1} \wedge x_{2,2} \wedge x_{4,2}$.
We have an action of $\GL(V^{\oplus d})$ on $\bigwedge^d(V^{\oplus d})$, which induces an action of $\GL(V)$ via the group homomorphism $g \mapsto \text{id}_d \otimes g$, i.e., sending matrices $g$ to block diagonal matrices that have $d$ many copies of $g$ on their main diagonal.
With this action of $\GL(V)$, the isomorphism in \eqref{eq:naturalisomgrassmannian} is an isomorphism of $\GL(V)$-representations.

For given $i,j$, $i \neq j$, we consider the Lie algebra element $\mathfrak g_{i,j} \in \mathfrak{gl}(V^{\otimes d})$ that is a block matrix as follows: the block matrix is zero everywhere but the block $(i,j)$ is the identity matrix on $V$.
Clearly applying $\mathfrak g_{i,j}$ is a $\GL(V)$-equivariant map.
Acting with $\mathfrak g_{i,j}$ on the tensor from \eqref{eq:longwedge} gives
\begin{equation}
\sum_{1 \leq a \leq \nu_i}
x_{k_{1,1},1} \wedge \cdots \wedge x_{k_{\nu_1,1},1}\wedge \cdots \wedge
x_{k_{1,i},[i=a,j,i]} \wedge \cdots \wedge x_{k_{\nu_i,i},[i=a,j,i]}
\wedge \cdots \wedge x_{k_{d,1},d} \wedge \cdots \wedge x_{k_{\nu_d,d},d},
\end{equation}
where
\[
[i\!=\!a,j,i] \ :=\ \begin{cases}
j & \text{ if } i=a \\
i & \text {otherwise.}
\end{cases}
\]
We denote by the map $\sigma_{i,j}$ the application of $\mathfrak g_{i,j}$. Using the canonical isomorphisms \eqref{eq:naturalisomgrassmannian} and \eqref{eq:Xnuwedges} we can write this more explicitly:
For $x \in \Lambda^{\alpha_i}$, $y \in \Lambda^{\alpha_j}$, $z \in \Lambda^{\alpha_k}$ have
\begin{eqnarray*}
\sigma_{i,k}(x\otimes y \otimes z) &=& \sum_{1 \leq a \leq \alpha_i} (-1)^{a+\alpha_i+\alpha_j} x_{-a} \otimes y \otimes (x_a \wedge z) \\
\sigma_{k,i}(x\otimes y \otimes z) &=& \sum_{1 \leq c \leq \alpha_k} (-1)^{c+\alpha_j+1} (x \wedge z_c) \otimes y \otimes z_{-c}
\end{eqnarray*}

Here for a tensor $x \in \bigwedge^d V$ we define $x_{-k}$ as the tensor obtained by ``removing the $k$-th tensor position'' (this is only well-defined if a basis of $V$ is fixed and an ordering of the basis vectors is fixed). For example, if $x=e_3\wedge e_1 \wedge e_4 \wedge e_3$, then $x_{-3} = e_3\wedge e_1 \wedge e_3$,
and $x_3 = e_4$.

Note that if $\la^*_j=0$, then $\sigma_{j,j+1}(\preschur^{\la^*}) = \{0\}$.
Define $\mc{I}^\la := \langle\sigma_{i,i+1}(Y_d) \cap \preschur^{\la^*}V \mid i\in \mathbb N\rangle$.
Note that if $\la^*_{j}>0$, then
$\sigma_{j,j+1}(\preschur^{\la^*}) \subseteq \preschur^{\la^*-e_j+e_{j+1}}$, hence
\begin{equation}\label{eq:ideal}
\mc{I}^\la = \langle\sigma_{i,i+1}(\preschur^{\la^*+e_i-e_{i+1}}) \mid \la^*_{i+1}>0\rangle.
\end{equation}

\begin{proposition}[{\cite[Cor.~1]{towber:1979:young_symmetry}}]\label{pro:schurquotient}
$\schur^λV$ and $\preschur^{\la^*} V \slash \mc{I}^λ.$ are isomorphic representations of $\GL(V)$.
The isomorphism maps each basis vector given by a semistandard tableau to a basis vector corresponding to the same semistandard tableau.
\end{proposition}

\begin{proposition}[{\cite[Thm.~2.5]{towber:1977:two_new_functiors_from_modules_to_algebras}}]\label{pro:schurvsweyl}
$\schur^λ V$ and $\weyl^λ V$ are isomorphic representations of $\GL(V)$.
\end{proposition}
The relation between the basis vectors corresponding to semistandard tableaux in Prop.~\ref{pro:schurvsweyl} is more involved than the straightforward relationship in Prop.~\ref{pro:schurquotient}.

\section{Young Flattenings in the Schur module basis}

We give an explicit description of the construction of the so-called \emph{Pieri inclusions} defined on the basis of Schur modules.
Olver \cite{olver:1982:differential_hyperform} first described the corresponding construction based on Weyl modules and we closely mimic this construction while taking care of the subtle signs that are introduced when using the Schur module basis.
To the best of our knowledge, this construction has never been explicitly described for Schur modules.
This algorithm will directly give the construction for Young flattenings.

\begin{figure}
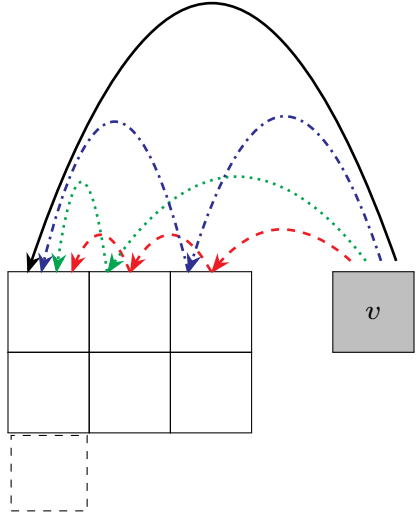

	\ytableausetup{centertableaux, mathmode, boxsize=3em}
	\ \\
	$$
	\begin{ytableau}
	\none \\
	\none \\
	\none \\
	\tikznode{b4}{\ }  & \tikznode{b3}{\ } & \tikznode{b2}{\ } & \none &*(gray!50) \tikznode{b1}{\scalebox{1.5}{$v$} } \\
	\ & \ &  \ \\
	\none[\begin{minipage}{1cm}\tikz{\draw[dashed] (0,0)--(1,0)--(1,1)--(0,1)--cycle;}
	      \end{minipage}]
	\end{ytableau}
	$$
	\tikz[overlay,remember picture]{%
		
		\draw[line width = 1, -Stealth, rotate=0] ([xshift=3mm,yshift=6mm]b1.north) parabola [parabola height=3.5cm] ([xshift=-2.8mm, yshift=6mm]b4);
		
		\draw[line width = 1, -Stealth, rotate=0, dashed, color=Red] ([xshift=-3mm,yshift=6mm]b1.north) parabola [parabola height=0.5cm] ([yshift=6mm]b2);
		\draw[line width = 1, -Stealth, rotate=0, dashed, color=Red] ([yshift=6mm]b2.north) parabola [parabola height=0.5cm] ([yshift=6mm]b3);
		\draw[line width = 1, -Stealth, rotate=0, dashed, color=Red] ([yshift=6mm]b3.north) parabola [parabola height=.5cm] ([xshift=3mm, yshift=6mm]b4);
		
		\draw[line width = 1, -Stealth, rotate=0, dotted, color=Green] ([xshift=-1mm,yshift=6mm]b1.north) parabola [parabola height=1.2cm] ([xshift=-3mm, yshift=6mm]b3);
		\draw[line width = 1, -Stealth, rotate=0, dotted, color=Green] ([xshift=-3mm, yshift=6mm]b3.north) parabola [parabola height=1.2cm] ([xshift=1mm, yshift=6mm]b4);
		
		\draw[line width = 1, -Stealth, rotate=0, dashdotted, color=Blue] ([xshift=1mm, yshift=6mm]b1.north) parabola [parabola height=2.0cm] ([xshift=-3mm, yshift=6mm]b2);
		\draw[line width = 1, -Stealth, rotate=0, dashdotted, color=Blue] ([xshift=-3mm, yshift=6mm]b2.north) parabola [parabola height=2.0cm] ([xshift=-1mm, yshift=6mm]b4);
	}
	\caption{The paths we need to traverse, when considering the Pieri inclusion from $(3,3)$ to $(3,3,1)$. Note that the boxes are moved in the direction of the arrows, but that the box movements are executed on each path starting with the \emph{leftmost} arrow.}
    \label{pic:idea}
\end{figure}
\ytableausetup{mathmode, boxsize=1em, centertableaux, smalltableaux}

Pieri's well-known formula states the following isomorphism of $\GL(V)$-representations:
$$\schur^d V \otimes \schur^{λ} V = \bigotimes_{\mu \vdash d+|\la|\atop μ/λ \in \ghs} \schur^μV.$$
The resulting $\GL(V)$-equivariant inclusions $$\varphi_{λ,μ}: \schur^dV \otimes \schur^λV \rightarrow \schur^μV$$ are called \emph{Pieri inclusions} and are unique up to scale by Schur's lemma.
We define $\varphi_{λ,μ}$ by composing Pieri inclusions for $d=1$ as follows.
For $\mu/\la \in \ghs$ let
$(\la=\la^{(0)},\la^{(1)},\la^{(2)},\ldots,\la^{(d)}=\mu)$ be the sequence of partitions obtained by adding one box at a time to $\la$ from left to right, so that after having added $d$ boxes we arrive at $\mu$.
Then define the $\GL(V)$-equivariant map $\psi_{\la,\mu}:V^{\otimes d}\otimes \schur^\la \to \schur^\mu$ via
\begin{equation}\label{eq:defpsi}
\psi_{\la,\mu} = \varphi_{\la^{(d-1)},\la^{(d)}}
\circ (\text{id}_{V}\otimes\varphi_{\la^{(d-2)},\la^{(d-1)}})
\circ \cdots \circ
(\text{id}_{V^{\otimes (d-2)}}\otimes\varphi_{\la^{(1)},\la^{(2)}})
\circ
(\text{id}_{V^{\otimes (d-1)}}\otimes\varphi_{\la^{(0)},\la^{(1)}}).
\end{equation}
By proving that the restriction of $\psi_{\la,\mu}$ to $\schur^d V \otimes \schur^\la V$ is nonzero (see Lemma~\ref{lem:flattening_non_zero}), it immediately follows that this restriction equals $\varphi_{\la,\mu}$ (up to a nonzero scalar).
It remains to describe $\varphi_{\la,\mu}$ for which $\mu/\la$ has only a single box and then prove Lemma~\ref{lem:flattening_non_zero}.

\subsection{The single box case}
We assume that $\mu/\la$ consists of a single box.

We will define the linear map $\zeta_{\la,\mu}: V \otimes \preschur^{\la^*} V \to \preschur^{\mu^*} V$,
show its $\GL(V)$-equivariance (\Cref{lem:equivariant}),
and then prove that it is well-defined on the quotient space $V \otimes \schur^\la V$ if we interpret its image in the quotient space $\schur^\mu V$ (\Cref{olver:thm}).
In this way, $\zeta_{\la,\mu}$ induces a map $\varphi_{\la,\mu}: V \otimes \schur^\la V \to \schur^\mu V$. By the uniqueness of the Pieri inclusions we have that this map equals $\varphi_{\la,\mu}$ or is the zero map (nonzeroness is proved in \Cref{lem:flattening_non_zero}).
Note that $V \otimes \preschur^{\la^*} V \simeq \preschur^{(\la^*,1)} V$ by definition.

Let $J = (J_1, \dots, J_p)$ be a finite sequence of positive integers. Define the linear map
\begin{equation}\label{def:sigmaJ}
\sigma_J: Y_d  \to Y_d ,\quad
\sigma_J \coloneqq \sigma_{J_1, J_2} \circ \sigma_{J_2, J_3} \circ \dots \circ \sigma_{J_{p-1}, J_p}.
\end{equation}
Pictorially, this means that a box is shifted from column $J_{p-1}$ to $J_p$, then from $J_{p-2}$ to $J_{p-1}$, and so on.
Clearly $\sigma_J$ is $\GL(V)$-equivariant as the composition of $\GL(V)$-equivariant maps.
Note that if $v \in \preschur^{(\la^*,1)}V$ and $\sigma_J$ moves a box from an empty column to another column, then $\sigma_J(v)=0$.

We consider the set of all strongly decreasing sequences of natural numbers from $m$ to $k$, which we denote by 
\[
\mathcal{A}^m_k := \{ J = (J_1, \dots, J_p) \mid p\in \NN, k = J_p < J_{p-1} < \dots < J_1 = m \}.
\]
For $\iota\geq k$ let $h_{k,\iota}:=\la^*_k-\la^*_\iota+\iota-k+1$ be the hook length in $\la$ of the box in column $k$ and row~$\iota$.
Let $k$ be the column where $\mu$ and $\la$ differ.
We define $D_J(\la^*)$ as the product of all hook lengths in $J$ with respect to column $k$, i.e., 
\begin{equation}
\label{eq:DJ}
D_J(\la^*) \coloneqq \prod_{q = 2}^{\abs{J}-1} h_{k,J_q}(\la^*) =  \prod_{q = 2}^{\abs{J} - 1} (\la^*_k - \la^*_{J_q} + J_q - k + 1).
\end{equation}
Finally, we define  $\zeta^μ_λ :Y_d \to Y_d$ by 

\[\zeta^μ_λ = \sum_{J \in \mathcal{A}^{λ_1+1}_k} \frac{\sigma_J}{D_J(λ^*)}.
\]
The map $\zeta^μ_λ$ maps $\preschur^{(\la^*,1)} V$ into $\preschur^{\mu^*} V$.

\begin{lemma}\label{lem:equivariant}
$\zeta^μ_λ: \preschur^{(\la^*,1)} V  \rightarrow \preschur^{\mu^*} V$ is $\GL(V)$-equivariant.
\end{lemma}
\begin{proof}
$\zeta^μ_λ$ is a linear combination of $\GL(V)$-equivariant maps.
\end{proof}

\begin{theorem}
\label{olver:thm}
Let $λ,μ$ be partitions with $μ$ being obtained from λ by appending a single box in row~$k$.
Then $\zeta_λ^μ(V \otimes \mc{I}_λ) \subseteq \mc{I}_μ$ and hence $\varphi_{\la,\mu}$ is well-defined.
\end{theorem}
Before proving Theorem~\ref{olver:thm} we first have to prove the following lemma.

We denote by $[f, g] := f g - g f$ the commutator of two linear maps $f$ and $g$.

\begin{lemma}
\label{lemma:weyl_commutator}

Let $v \in \preschur^\alpha V$ for some column lengths $\alpha$ with $\alpha_i\neq 0$ and $\alpha_k\neq 0$ and let $i \neq j, k \neq l \in \NN$. Then,
\begin{align*}
[\sigma_{i ,j}, \sigma_{k,l}] (v) = 
\begin{cases}
(\alpha_j - \alpha_i) v & \text{if } i = l,j = k,\\
 \sigma_{k,j} (v) & \text{if } i = l, j \neq k,\\
-\sigma_{i,l} (v) & \text{if } i \neq l, j = k,\\
0 & \text{otherwise.}
\end{cases}
\end{align*}
\end{lemma}
Note the similarity to Lemma 5.4 in \cite{olver:1982:differential_hyperform} with the exception that the sign in the first case is reversed. Moreover, \cite{olver:1982:differential_hyperform} ignores handling the special case when column lengths vanish. We handle these cases explicitly. If $\la^*_{j}=0$, $\la^*_{i}\neq 0$ and $v \in \preschur^{\la^*}$, then
\begin{equation}\label{eq:specialzero}
\sigma_{j,i}\sigma_{i,j}(v) = \la^*_{i} v.
\end{equation}
Moreover, if 
$\la^*_{j}=0$, $\la^*_{i}\neq 0$, $i\neq k$ and $v \in \preschur^{\la^*}$, then
\begin{equation}\label{eq:specialzeroII}
\sigma_{j,k}\sigma_{i,j}(v) = \sigma_{i,k}(v).
\end{equation}

\begin{proof}[Proof of Lemma~\ref{lemma:weyl_commutator}]
We focus on the key positions in the tensor.
\begin{itemize}
\item $i = l, j = k$: $x\in\bigwedge^{\alpha_i}V$, $y\in\bigwedge^{\alpha_j}V$.
\begin{align*}
[σ_{i,j} , σ_{j, i}](x \otimes y) &= σ_{i,j} \Big(\sum_{1 \leq b \leq \alpha_j} (-1)^{b+1} (x\wedge y_b) \otimes y_{-b} \Big)  -  σ_{j, i}  \Big(\sum_{1 \leq a \leq \alpha_i} (-1)^{a+\alpha_i} x_{-a} \otimes (x_a \wedge y)\Big)
\\
&= \sum_{1 \leq a \leq \alpha_i + 1 \atop 1 \leq b \leq \alpha_j}(-1)^{a+b+\alpha_i} (x\wedge y_b)_{-a} \otimes ((x \wedge y_b)_{a}\wedge y_{-b} )  \\
&\phantom{= }- \sum_{1 \leq a \leq \alpha_i \atop 1 \leq b \leq \alpha_j} (-1)^{a+b+1+\alpha_i} ( x_{-a}\wedge (x_a \wedge y)_b ) \otimes (x_a \wedge y)_{-b}
\\
&= \sum_{1 \leq a \leq \alpha_i \atop 1 \leq b \leq \alpha_j} (-1)^{a+b+\alpha_i} (x_{-a} \wedge y_b) \otimes (x_a \wedge y_{-b}) \\
&\phantom{= }- \sum_{2 \leq b \leq \alpha_{j}+1 \atop 1 \leq a \leq \alpha_i} (-1)^{a+b+\alpha_{i}+1} (x_{-a} \wedge y_{b-1}) \otimes (x_a \wedge y_{-(b-1)}) \\
&\phantom{= }+ \sum_{1 \leq b \leq \alpha_j} (-1)^{b+1} x \otimes (y_b \wedge y_{-b})\\
&\phantom{= }- \sum_{1 \leq a \leq \alpha_i} (-1)^{a+\alpha_i} (x_{-a} \wedge x_a) \otimes y \\
&= \sum_{1\leq b \leq \alpha_j} x \otimes y - \sum_{1 \leq a \leq \alpha_i} x \otimes y = (\alpha_j - \alpha_i) (x \otimes  y)
\end{align*}
\item $i = l, j \neq k$: $x\in\bigwedge^{\alpha_i}V$, $y\in\bigwedge^{\alpha_j}V$, $z\in\bigwedge^{\alpha_K}V$.
\begin{align*}
[σ_{i,j} , σ_{k, i}](x \otimes y \otimes z) &= σ_{i,j} \Big(\sum_{1 \leq c \leq \alpha_k} (-1)^{\alpha_j+c+1} (x\wedge z_c) \otimes y \otimes z_{-c}\Big) \\
&\phantom{=}-  σ_{k, i}  \Big(\sum_{1 \leq a \leq \alpha_i} (-1)^{\alpha_i+a} x_{-a} \otimes (x_a \wedge y) \otimes z\Big) \\
&=\sum_{1 \leq c \leq \alpha_k \atop 1 \leq a \leq \alpha_i+1} (-1)^{\alpha_i+\alpha_j+a+c} (x\wedge z_c)_{-a} \otimes ((x\wedge z_c)_{a}) \wedge y  \otimes z_{-c} \\
&\phantom{=}  -\sum_{1 \leq a \leq \alpha_i \atop 1 \leq c \leq \alpha_k} (-1)^{a+c+\alpha_i+\alpha_j} (x_{-a}\wedge z_c) \otimes (x_a \wedge y) \otimes z_{-c})\\
&= \sum_{1 \leq a \leq \alpha_i \atop 1 \leq c \leq \alpha_k} (-1)^{\alpha_i+\alpha_j+a+c} (x_{-a}\wedge z_c) \otimes (x_a \wedge y) \otimes z_{-c} )  \\
&\phantom{=}  -\sum_{1 \leq a \leq \alpha_i \atop 1 \leq c \leq \alpha_k} (-1)^{\alpha_i+\alpha_j+a+c} (x_{-a}\wedge z_c) \otimes (x_a \wedge y) \otimes z_{-c} )\\
&\phantom{=} + \sum_{1 \leq c \leq \alpha_k} (-1)^{\alpha_j + c +1} x \otimes (z_c \wedge y) \otimes z_{-c}
\\
&= \sum_{1 \leq c \leq \alpha_k} (-1)^{c +1} x \otimes (y \wedge z_c) \otimes z_{-c}
\\
&= σ_{k, j} (x\otimes y \otimes z)
\end{align*}
\item $i \neq l, j = k$:
Equivalent to the case $i = l, j \neq k$ but changing the order of elements in the commutator. Thus, the sign changes:
\begin{align*}
        [σ_{i,j} , σ_{j, l}](x \otimes y \otimes z)  = - [σ_{j,l} , σ_{i, j}](x \otimes y \otimes z) = -σ_{i,l} (x\otimes y \otimes z)
    \end{align*}
\item $i \neq l, j \neq k$:
If all $i,j,k,l$ are pairwise distinct, then both maps affect distinct columns and hence they commute.
We first treat the case $j=l \notin \{i,k\}$, $i \neq k$.
\begin{align*}
[σ_{i,k} , σ_{j,k}](x \otimes y \otimes z) &= σ_{i,k} \Big( \sum_{1 \leq b \leq \alpha_j}(-1)^{\alpha_j+b} x \otimes y_{-b}\otimes(y_b\wedge z) \Big) \\
&\phantom{=}- σ_{j,k}\Big(\sum_{1\leq a \leq \alpha_i}(-1)^{\alpha_i+\alpha_j+a}x_{-a}\otimes y \otimes (x_a\wedge z)\Big) \\
&= \sum_{1 \leq a \leq \alpha_i \atop 1 \leq b \leq \alpha_j} (-1)^{\alpha_j+b+a+\alpha_i+\alpha_j+1} x_{-a}\otimes y_{-b} \otimes (x_a \wedge y_b \wedge z) \\
&\phantom{=}-\sum_{1 \leq a \leq \alpha_i \atop 1 \leq b \leq \alpha_j}(-1)^{\alpha_j+b+a+\alpha_i+\alpha_j} x_{-a}\otimes y_{-b} \otimes (y_b \wedge x_a \wedge z)
\\ &= 0
\end{align*}
We now treat the remaining case $i=k \notin \{j,l\}$, $j \neq l$. For a basis vector $x$ let $x_{-\{a,b\}}$ denote the basis vector with positions $a$ and $b$ removed.
\begin{align*}
[σ_{i,j} , σ_{i,l}](x \otimes y \otimes z) &= σ_{i,j} \Big(\sum_{1 \leq a \leq \alpha_i} (-1)^{\alpha_j+a+\alpha_i} x_{-a} \otimes y \otimes (x_a\wedge z)\Big) \\
&\phantom{=}-  σ_{i,l}  \Big(\sum_{1 \leq a \leq \alpha_i} (-1)^{\alpha_i+a} x_{-a} \otimes (x_a \wedge y) \otimes z\Big) \\
&=\sum_{1 \leq a \leq \alpha_i \atop 1 \leq a' \leq \alpha_i-1} (-1)^{a+a'+\alpha_j+1} (x_{-a})_{-a'} \otimes ((x_{-a})_{a'} \wedge y) \otimes (x_a \wedge z) \\
&\phantom{=}  -\sum_{1 \leq a \leq \alpha_i \atop 1 \leq a' \leq \alpha_i-1} (-1)^{a+a'+\alpha_j} (x_{-a})_{-a'} \otimes (x_a \wedge y) \otimes ((x_{-a})_{a'}\wedge z)\\
&= \sum_{1 \leq a,\tilde a \leq \alpha_i \atop a \neq \tilde a} (-1)^{a+\tilde a +\alpha_j +1 + [\tilde a > a]} x_{-\{a,\tilde a\}} \otimes (x_{\tilde a} \wedge y) \otimes (x_a \wedge z)  \\
&\phantom{=}  -\sum_{1 \leq a,\tilde a \leq \alpha_i \atop a \neq \tilde a} (-1)^{a+\tilde a + \alpha_j + [\tilde a > a]} x_{-\{a,\tilde a\}} \otimes (x_a \wedge y) \otimes (x_{\tilde a}\wedge z)\\
&= 0
\end{align*}
where $[b>a]$ is 1 if $b>a$ and 0 otherwise. Here we used the notation
\[
\tilde a = \begin{cases}
a' &\text{ if } a'<a \\
a'+1 &\text{ if } a'\geq a
\end{cases}.
\]
Note that the second case happens exactly when $\tilde a > a$.\qedhere
\end{itemize}
\end{proof}

We will make heavy use of the following identity:
Let $σ_1,σ_2, σ_3$ be linear maps, then
\begin{equation}\label{eq:Leibniz}
[σ_1 \circ σ_2, σ_3] = σ_1 \circ [σ_2, σ_3]  + [σ_1, σ_3] \circ σ_2.
\end{equation}
The rule can be interpreted as the Leibniz rule for $ad_{A}(B) = [A, B]$. 

\begin{proof}[Proof of Theorem~\ref{olver:thm}]
    Using \eqref{eq:ideal} we see that it suffices to prove that if
$w \in \sigma_{i,i+1}(\preschur^{(\la^*+e_i-e_{i+1},1)})$, then $\zeta_{\la,\mu}(w) \in \sigma_{i,i+1}(\preschur^{\mu^*+e_i-e_{i+1}})$. Let $w = \sigma_{i,i+1}(v)$ with $v \in\preschur^{\la^*+e_i-e_{i+1}}$.
We now show that $\zeta_{\la,\mu}(\sigma_{i,i+1}(v))=\sigma_{i,i+1}(\zeta_{\la,\mu}(v))$, which finishes the proof.

We split the proof according to the different relations of $i$ and $k$: 
\begin{itemize}
\item $i < k$: 
If $i+1 < k$, then clearly $[σ_A, σ_{i, i + 1}] = 0$ for every $A \in \mathcal{A}_k^{\la_1+1}$, because $A \cap \{i,i+1\}=\emptyset$.
Consider the case $i+1 = k$. Every $A \in \mathcal{A}_k^{\la_1+1}$ can be written as $(B,i+1)$ with $B \in \mathcal{A}^{\la_1+1}_m$ for some $m>i+1$.
\[
[\sigma_{A},\sigma_{i,i+1}] \stackrel{B\cap\{i,i+1\}=\emptyset}{=} \sigma_{B}[\sigma_{m,i+1},\sigma_{i,i+1}] \stackrel{\text{\Cref{lemma:weyl_commutator}}}{=} 0.
\]
\item $i = k$:
We divide the sequences in $\mathcal{A}_k^{λ_1 + 1}$ (summed up over in $\zeta_λ^μ$) as follows: For every  $m > k + 1$ and $B \in \mathcal{A}_m^{\la_1+1}$, let either $A_2 = (B, k+1, k)$ or $ A_1 = (B, k)$.
In fact, the sequences come in pairs. Adding/removing the entry $k+1$ maps the elements of the pairs to each other.

$$ [\sigma_{A_1}, \sigma_{k, k+1}] (v) \stackrel{B \cap \{k,k+1\} = \emptyset}{=} \sigma_B \circ [\sigma_{m, k}, \sigma_{k,k+1} ] (v) \stackrel{\text{\Cref{lemma:weyl_commutator}}}{=} -\sigma_B \circ \sigma_{m, k+1}(v),$$
as well as (if $\la^*_{k+1}>1$):
\begin{eqnarray*}
[\sigma_{A_2}, \sigma_{k, k+1}](v) &\stackrel{B \cap \{k,k+1\} = \emptyset}{=}&  \sigma_B \circ [\sigma_{m, k+1} \circ \sigma_{k+1, k}, \sigma_{k,k+1} ] (v)\\
&\stackrel{\eqref{eq:Leibniz}}{=}& \sigma_B \circ (\sigma_{m, k+1} [\sigma_{k+1, k}, \sigma_{k,k+1}] + [\sigma_{m, k+1}, \sigma_{k,k+1}] \circ \sigma_{k+1, k}) (v) \\
&\stackrel{\text{\Cref{lemma:weyl_commutator}}}{=}& \sigma_B \circ (\sigma_{m, k+1} [\sigma_{k+1, k}, \sigma_{k,k+1}])(v)\\
&\stackrel{\text{\Cref{lemma:weyl_commutator}}, v \in \preschur^{\la^*+e_i-e_{i+1}}}{=}& ((λ^*_k + 1)-(λ^*_{k+1} - 1)) \sigma_B \circ \sigma_{m, k+1}(v)\\
&=& (λ^*_k - λ^*_{k+1} + 2) \sigma_B \circ \sigma_{m, k+1}(v)
\end{eqnarray*}
If $\la^*_{k+1}=1$, then the same is true:
\begin{eqnarray*}
[\sigma_{A_2}, \sigma_{k, k+1}](v) &=& \sigma_B \circ (\sigma_{m, k+1} [\sigma_{k+1, k}, \sigma_{k,k+1}])(v)\\
&\stackrel{v \in \preschur^{\la^*+e_k-e_{k+1}}}{=}& \sigma_B \circ \sigma_{m, k+1}\circ \sigma_{k+1, k}\circ\sigma_{k,k+1}(v)\\
&\stackrel{\eqref{eq:specialzero}}{=}& (\la_k^*+1) \sigma_B \circ \sigma_{m, k+1}(v) \\
&=& (λ^*_k - λ^*_{k+1} + 2) \sigma_B \circ \sigma_{m, k+1}(v)\\
\end{eqnarray*}

We take a weighted sum of two paired up sequences:
\begin{eqnarray*}
[D^{-1}_{A_1} \sigma_{A_1} + D^{-1}_{A_2} \sigma_{A_2}, \sigma_{k, k+1}](v)
&=& \left(-D^{-1}_{A_1} + D^{-1}_{A_2} (λ^*_k - λ^*_{k+1} + 2)\right) \sigma_{B}\circ \sigma_{m, k+1}(v) \\
&=& D^{-1}_{A_1} \left(-1 + \frac{λ^*_k - λ^*_{k+1} + 2} {λ^*_k - λ^*_{k+1} + 2} \right) \sigma_{B}\circ \sigma_{m, k+1}(v) \\
&=& 0
\end{eqnarray*} 
since $D_{A_2} = D_{A_1} (λ^*_k - λ^*_{k+1} + 2)$,
because $h_{k,k+1}(\la^*)=λ^*_k - λ^*_{k+1} + 2$ by \eqref{eq:DJ}. 
Since the weighted sum over two paired up sequences yields zero, the weighted sum over all sequences in $\mathcal{A}_k^{λ_1 + 1}$ yields zero.

\item $i > k$:
Again, we divide the sequences in $\mathcal{A}_k^{λ_1+1}$: For every $B \in \mathcal{A}_k^{m_1}$ and $C \in \mathcal{A}_{m_2}^{λ_1+1}$, we have $A_1 = (C,i,B)$, $A_2 = (C,i+1,B)$, and $A_3 = (C, i+1, i, B)$ for $B \in \mathcal{A}_k^{m_1}$ and $C \in \mathcal{A}_{m_2}^{\la_1+1}$ for some $m_1 < i$ and $m_2 > i+1$.
This time the sequences come in quadruples of sequences that can be obtained from each other by adding/removing $i$ and $i+1$.
Clearly $[\sigma_{A_0},\sigma_{i,i+1}](v)=0$, because $(B\cup C)\cap \{i,i+1\}=\emptyset$. So these sequences contribute zero to the sum $[\zeta_\la^\mu,\sigma_{i,i+1}]$.
We ignore these sequences and are left with triples of sequences instead of quadruples.

We have 
\begin{eqnarray*}
[σ_{A_1}, σ_{i, i+1}] (v) &\stackrel{\{i,i+1\}\cap(B\cup C)=\emptyset}{=}& σ_C \circ ([σ_{m_2, i} \circ  σ_{i, m_1}, σ_{i,i+1} ]) \circ σ_B (v)\\ 
&\stackrel{\eqref{eq:Leibniz}}{=}&  σ_C \circ ( σ_{m_2, i} \circ [σ_{i, m_1}, σ_{i,i+1} ] + [σ_{m_2, i}, σ_{i,i+1}] \circ  σ_{i, m_1}) \circ σ_B (v)\\ 
&\stackrel{\text{\Cref{lemma:weyl_commutator}}}{=}& - σ_C \circ σ_{m_2, i+1}\circ  σ_{i, m_1} \circ  σ_B (v)
\end{eqnarray*}
If $\la^*_{i+1}>1$:
\begin{eqnarray*}
[σ_{A_2}, σ_{i, i+1}] (v) 
&\stackrel{\{i,i+1\}\cap(B\cup C)=\emptyset}{=}& σ_C \circ [σ_{m_2, i+1} \circ σ_{i+1, m_1}, σ_{i,i+1} ]\circ σ_B (v) \\
&\stackrel{\eqref{eq:Leibniz}}{=}& σ_C \circ (σ_{m_2, i+1} \circ [σ_{i+1, m_1}, σ_{i,i+1}]+[σ_{m_2, i+1} \circ σ_{i,i+1}]\circ σ_{i+1, m_1}) \circ σ_B (v) \\
&\stackrel{\text{\Cref{lemma:weyl_commutator}}}{=}& σ_C \circ σ_{m_2, i+1} \circ [σ_{i+1, m_1}, σ_{i,i+1} ]\circ σ_B (v) \\
&\stackrel{\text{\Cref{lemma:weyl_commutator}}, \la_{i+1}^*>1}{=}&
σ_C \circ σ_{m_2, i+1} \circ σ_{i, m_1}\circ  σ_B (v)
\end{eqnarray*}
We get the same result for the case $\la^*_{i+1}=1$:
\begin{eqnarray*}
[σ_{A_2}, σ_{i, i+1}] (v) 
&=& σ_C \circ σ_{m_2, i+1} \circ [σ_{i+1, m_1}, σ_{i,i+1} ]\circ σ_B (v) \\
&=& σ_C \circ σ_{m_2, i+1} \circ σ_{i+1, m_1} \circ σ_{i,i+1} \circ σ_B (v) \\
&\stackrel{\eqref{eq:specialzeroII}}{=}&
σ_C \circ σ_{m_2, i+1} \circ σ_{i, m_1}\circ  σ_B (v)
\end{eqnarray*}

If $\la^*_{i+1}>1$:
\begin{eqnarray*}
&&[σ_{A_3}, σ_{i, i+1}](v)\\
&\stackrel{\{i,i+1\}\cap(B\cup C)=\emptyset}{=}& σ_C  \circ [σ_{m_2, i+1} \circ σ_{i+1, i} \circ σ_{i, m_1}, σ_{i,i+1} ]\circ σ_B (v) \\
&\stackrel{\eqref{eq:Leibniz}}{=}& σ_C  \circ(σ_{m_2, i+1} \circ σ_{i+1, i} \circ [σ_{i, m_1}, σ_{i,i+1}]+[σ_{m_2, i+1} \circ σ_{i+1, i},σ_{i,i+1}]\circ σ_{i, m_1})\circ σ_B (v) \\
&\stackrel{\text{\Cref{lemma:weyl_commutator}}}{=}&
σ_C  \circ([σ_{m_2, i+1} \circ σ_{i+1, i},σ_{i,i+1}]\circ σ_{i, m_1})\circ σ_B (v) \\
&\stackrel{\eqref{eq:Leibniz}}{=}&
σ_C  \circ (
σ_{m_2, i+1} \circ [σ_{i+1, i}\circσ_{i,i+1}] + [σ_{m_2, i+1},σ_{i,i+1}]\circ σ_{i+1, i})\circ σ_{i, m_1} \circ σ_B (v) \\
&\stackrel{\text{\Cref{lemma:weyl_commutator}}}{=}&
σ_C \circ (
σ_{m_2, i+1} \circ [σ_{i+1, i}\circσ_{i,i+1}])\circ σ_{i, m_1} \circ σ_B (v) \\
&\stackrel{\text{\Cref{lemma:weyl_commutator}}, v \in \preschur^{\la^*+e_i-e_{i+1}}}{=}&
(λ^*_i - (λ^*_{i+1} - 1)) σ_C  \circ σ_{m_2, i+1} \circ σ_{i, m_1}\circ  σ_B (v).
\end{eqnarray*}

Note that in the last equation we used that $(σ_{i, m_1} \circ σ_B) (v) \in \preschur^\nu V$ with $\nu_i=\la^*_i$ and $\nu_{i+1}=\la^*_{i+1}-1$, because $v \in \preschur^{\la^*+e_i-e_{i+1}}$ and $B \cap \{i,i+1\}=\emptyset$.

We get the same result for the case $\la^*_{i+1}=1$:
\begin{eqnarray*}
&&[σ_{A_3}, σ_{i, i+1}](v)\\
&=&σ_C \circ (
σ_{m_2, i+1} \circ [σ_{i+1, i}\circσ_{i,i+1}])\circ σ_{i, m_1} \circ σ_B (v) \\
&\stackrel{\eqref{eq:specialzero}}{=}& \la_i^* σ_C  \circ σ_{m_2, i+1} \circ σ_{i, m_1}\circ  σ_B (v)
\end{eqnarray*}

We take a weighted sum of a triple of grouped sequences:
\begin{align*}
&[D^{-1}_{A_1} σ_{A_1} + D^{-1}_{A_2} σ_{A_2} - D^{-1}_{A_3} σ_{A_3}, σ_{i, i+1}](v) = 0,
\end{align*}
which can be seen as follows:
$D_{A_3} = D_{A_1} h_{k,{i+1}(\la^*)}$ and $D_{A_3} = D_{A_2} h_{k,{i+1}(\la^*)}$
implies
\begin{eqnarray*}
\frac{-1}{D_{A_1}} + \frac{1}{D_{A_2}} + \frac{\la_i^*-\la_{i+1}^*+1}{D_{A_3}}
&=&\\
\frac{1}{D_{A_3}} \left( -h_{k,{i+1}}(\la^*) + h_{k_i}(\la^*) + \la_i^* - \la_{i+1}^* + 1  \right)
&=& - \la_i^* + \la_{i+1}^* -1 + \la_i^* - \la_{i+1}^* + 1 = 0. \qedhere
\end{eqnarray*}
\end{itemize}
\end{proof}

\subsection{Nonzeroness}

Next, we show that the map $φ_λ^μ$ is nonzero.
Let $Z^\la$ denote the Young tableau of shape $\la$ in which each box in row $i$ has entry $i+1$, with the exception that columns with $\dim V$ many boxes just have entries $1,2,\ldots,\dim V$ from top to bottom.
Recall the definition of $\psi_\la^\mu$ from \eqref{eq:defpsi}.
\begin{lemma}
\label{lem:flattening_non_zero}
Let $λ, μ$ be partitions with $\la\subseteq\mu$, $\mu/\la\in\ghs$.
Then the restriction of $\psi_λ^μ$ to $\schur^d V \otimes \schur^\la V$ is nonzero. More precisely,
$\psi_λ^μ(v_{1}^{\otimes (|\mu|-|\la|)} \otimes Z^\la)$ is nonzero.
\end{lemma}
\begin{proof}
Let $k$ denote the smallest column index in which $\la$ and $\mu$ differ.
Since in all columns to the right of column $k$ we only have numbers that appear in column $k$, the only
transition sequence in $A^{λ_{1}+1}_{k}$ which does not vanish is $J = (k, λ_1 + 1)$.
Hence the image of $φ_λ^{\la'}$ is a single tableau which is either of shape $\mu$, or for the first column $k'$ in which $\mu$ and $\la'$ differ we have that all columns right of column $k'$ only have entries that occur in column $k'$.
Therefore, again there is only one transition sequence.
We continue this and end up with a tableau that differs from $Z^\la$ by having additional entries 1 in each column where $\mu$ and $\la$ differ.
This tableau does not have a repeated entry in any column, so straightening this tableau does not result in the zero vector \cite{hodges:2020}, which finishes the proof.
\end{proof}

In the proof of Lemma~\ref{lem:flattening_non_zero} we used the exact order of maps in \eqref{eq:defpsi}.
The following small argument shows that this order does not matter.
\begin{claim}\label{cla:youngslattice}
If the boxes in \eqref{eq:defpsi} are added in any other order, then we get the same map up to a nonzero scalar.
\end{claim}
\begin{proof}
First, we can see that the proof in Lemma~\ref{lem:flattening_non_zero} can be adapted to show nonzeroness for different orders.
Indeed, if we add the boxes in a different order, then more transition sequences have to be considered, but
the only relevant transition sequences all end up with the same tableau and they all give a positive contribution to the end result, so nothing cancels out.

Since $\varphi_{\la,\mu}$ maps $\schur^d V \otimes \schur^\la V$ to $\schur^\mu V$ and the right-hand side is irreducible and the left-hand side contains a single copy of $\schur^\mu V$, Schur's lemma implies that all such maps are the same up to scale.
\end{proof}

\section{Software}
The Pieri inclusion
$$\varphi_{λ,μ}: \schur^dV \otimes \schur^λV \rightarrow \schur^μV$$
induces a linear map
\begin{equation}\label{eq:defF}
\mc{F}_{λ,μ} : \schur^dV \to \text{End}(\schur^λV,\schur^μV)
\end{equation}
For a homogeneous polynomial $p$ we are interested in the rank of the image $\mc{F}_{λ,μ}(p)$.
Analogously for Weyl modules.

In this section we provide a small example for computing Young flattenings in the basis of Schur and Weyl modules, using our implementation and Sam's Macaulay2 implementation. In both cases we will be working over $V = \CC^3$ with a basis $\{a,b,c\}$. We will consider the shapes $λ = (2,1,1)$ and $μ = (5,2,1)$ and are interested in the rank of the respective flattening of the following polynomial 
    $$p = a^3 + b c^2 \in \QQ[a,b,c]_{=3} \cong \schur^{(3)}V.$$
In other words we search for 
    \[\rank(\mc{F}_{(5,2,1), (2,1)}(a^3 + b c^2)).\]
First, let us take a look at Macaulay2.

\subsection{An example of \pierimaps}
\label{subsec:example_pieri}
To the best of our knowledge, the first use of Macaulay2 to compute the rank of Young flattenings together with code examples was given by Oeding~\cite{oeding:2016:border_ranks_of_monomials_v1}.

The following command loads the \pierimaps package:
\begin{lstlisting}
loadPackage "PieriMaps"
\end{lstlisting}

We can define a polynomial ring over the rational numbers and define the polynomial $p$ as follows:
\begin{lstlisting}
R = QQ[a,b,c]
p = a^3 + b c^2
\end{lstlisting}

The function \texttt{pieri} available in \pierimaps computes the polarization map 
	$$ \mc{P}_{μ, λ}^\weyl: \weyl^μV \rightarrow \weyl^{(d)}V \otimes \weyl^λV,$$
which was also explicitly described by Olver~\cite{olver:1982:differential_hyperform}. 
It is the dual of the Pieri inclusion.

The function \texttt{pieri} takes three arguments: the dominating tableau $μ$;  a list of $d$ row indices $r = (r_1, \dots, r_d)$, such that λ can be obtained from μ by deleting the last box in row $r_1$, then the last box in row $r_2$, and so on; and the underlying polynomial ring $\QQ[a,b,c]$.

The following code computes a matrix representing $\mc{P}^\weyl_{(5,2,1), (4,1)}$. Note that to go from $(5,2,1)$ to $(4,1)$ we have to remove a box in row 1, then from row 2, and finally from row 3.
The following command corresponds to this operation:

\begin{lstlisting}
MX = pieri({5,2,1}, {1,2,3}, R)
\end{lstlisting}
which outputs a $24 \times 24$ matrix of homogeneous degree 3 polynomials in $a$, $b$, $c$.
We differentiate the entries by $p$ and compute the rank
\begin{lstlisting}
rank(diff(p, MX))
\end{lstlisting}
which gives 18.
It follows that $\rank(\mc{P}_{λ, μ}^\weyl(p)) = 18$.

\subsection{Using our software}
Our implementation is available as ancillary files to this paper. The README file contains detailed installation instructions.
For the rank computation we rely on the linear algebra implementation of Macaulay2, so we assume that Macaulay2 is installed.

After the installation, our tool is called from the command line with 4 parameters:
\begin{enumerate}
 \item The number of variables $n$,
 \item the partition $\mu$,
 \item the list of row indices $(r_1,\ldots,r_d)$ from which boxes are to be removed to obtain $\lambda$,
 \item and the polynomial $p$.
\end{enumerate}
The example from the previous sections is calculated via
\begin{lstlisting}[style=mystyle_small, basicstyle=\fontsize{9}{13}\selectfont\ttfamily\color{Blue},]
./flattening 3 [5,2,1] [1,2,3] a^3+b*c^2
\end{lstlisting}
which also outputs 18.

\subsection{Running time comparison}
We compared our implementation to the \pierimaps package.
Our implementation does not have multi-processor support.
The computations were run on a laptop, quad-core i5-6200U CPU with 2.30GHz with 8GB of memory.
On this fairly weak machine, the larger examples from \cite{oeding:2019:border_ranks_of_monomials} crash \pierimaps.
Our software constructs flattening matrices for each monomial and adds them up, so for a fairer comparison we used random dense polynomials $p$ that were generated with the following \texttt{Sagemath} code (adjust the degree and the number of variables appropriately):
\begin{lstlisting}[style=mystyle_small, basicstyle=\fontsize{7}{9}\selectfont\ttfamily\color{Blue},]
var('a,b,c,d,e')
R=ZZ[a,b,c,d,e]
str(sum([R.random_element(degree=4,terms=Infinity)
  .homogenize(var=randint(0,len(R.gens())-1)) for _ in range(10)])).replace(" ","")
\end{lstlisting}
We chose six quite different examples in our comparison.

\begin{center}
\begin{tabular}{c|c|c|c||c|c}
$n$ & $\mu$ & $r$ & $p$ & \pierimaps & Our software\\
\hline
5 & [4,4,4,4] & [4,4,4,4] & random & 1m\,18s & 6s \\
5 & [5,3,1] & [1,2,3] & random & 5m\,19s & 2s \\
5 & [7,5,4,3,2] & [1,1,2,3,4,5,5] & $x_0^2 x_1^2 x_2 x_3 x_4$ & 59m\,22s & 1m \\
5 & [7,5,4,3,2] & [1,1,2,3,4,5,5] & $x^7$ & 59m\,47s & 1s \\
4 & [7,5,3,1] & [1,2,3,4] & random & 2h\,35m\,3s & 9s\\
6 & [7,5,4,3,2,1] & [1,1,2,3,4,5,6] & $x^7$ & $>$11h\footnotemark & 16s
\end{tabular}\footnotetext{We manually terminated the computation after 11 hours.}
\end{center}
We observe that our software is much faster than the \pierimaps implementation.
The most extreme boost (an improvement factor of over 1000) is obtained when flattening the important case $x^d$. This rank is used in the denominator of \eqref{eq:rankquotient}.
The use of the Schur basis allows us to stop a computation path for a tableau as soon as the tableau has a double entry in a column. This speeds up the computation. But we think that most of the speed-up comes from the fact that we circumvent the construction of the parameterized flattening matrix and construct the matrix directly.

\section{The oversimplification in the literature}\label{sec:oversimplification}

In \cite{farnsworth:2015:koszul-young_flattenings} (Section 5) and \cite{oeding:2016:border_ranks_of_monomials_v1} (Def.~3.2)
the Young flattening is described with an simple procedure that we call the \emph{box-filling flattening}.
We present it here and give a small counterexample to its equivariance.

We define the box-filling flattening $\mc{F}^{\mathit{fill}}_{λ,μ}: \schur^{(d)}V \rightarrow \Hom(\schur^λV, \schur^μV)$ for monomials and use linear continuation.
For $α \in \NN^n$ we use the notation $x^α\coloneqq \prod^n_{i = 1}x_i^{α_i}$. 
Let $x^α\in \schur^{(d)}V$. Let $T$ be a semistandard Young tableau of shape λ with entries from $\{1,\ldots,\dim V\}$. These $T$ with the Grassmann-Pl\"ucker relations form a basis of $\schur^λV$.
Define
\[\mc{F}^{\mathit{fill}}_{λ, μ}(x^α)(T) = \sum_{S}S,\]
where the sum is over all $S$ of shape $\mu$ that can be obtained from $T$ by adding boxes to $T$: exactly $\alpha_i$ boxes with number $i$.
The resulting $S$ may not be semistandard, but can be expressed over the basis of semistandard tableaux via straightening.

Let $λ = (2)$, $μ = (2,1)$, $V = \CC^2$, $g = \begin{pmatrix} 0 & 1 \\ 1 & 0 \end{pmatrix} \in \GL_2$.

The matrices corresponding to the linear maps $\mc{F}^{\mathit{fill}}_{(2), (2,1)}(x_1)$ and $\mc{F}^{\mathit{fill}}_{(2), (2,1)}(x_2)$ are
\[
\begin{tikzpicture}
\node at (0,0) {
\begin{tabular}{c|ccc}
$\mc{F}^{\mathit{fill}}_{(2), (2,1)}(x_1)$ & \ytableaushort{11} & \ytableaushort{12} & \ytableaushort{22} \\
\hline
\tikz{\node at (0,0) {\ytableaushort{11,2}};} & \raisebox{0.2cm}0 & \raisebox{0.2cm}0 & \raisebox{0.2cm}0 \\
\tikz{\node at (0,0) {\ytableaushort{12,2}};} & \raisebox{0.2cm}0 & \raisebox{0.2cm}0 & \raisebox{0.2cm}{$-1$} \\
\end{tabular}
};
\node at (8,0) {
\begin{tabular}{c|ccc}
$\mc{F}^{\mathit{fill}}_{(2), (2,1)}(x_2)$ & \ytableaushort{11} & \ytableaushort{12} & \ytableaushort{22} \\
\hline
\tikz{\node at (0,0) {\ytableaushort{11,2}};} & \raisebox{0.2cm}1 & \raisebox{0.2cm}0 & \raisebox{0.2cm}0 \\
\tikz{\node at (0,0) {\ytableaushort{12,2}};} & \raisebox{0.2cm}0 & \raisebox{0.2cm}1 & \raisebox{0.2cm}0 \\
\end{tabular}
};
\end{tikzpicture}
\]
but applying $g$ to the matrix on the left-hand side (this means swapping column 1 and 3, swapping row 1 and 2, and inverting the sign) does not yield the matrix on the right-hand side, but instead yields
\[
\begin{tabular}{c|ccc}
$g\mc{F}^{\mathit{fill}}_{(2), (2,1)}(x_1)$ & \ytableaushort{11} & \ytableaushort{12} & \ytableaushort{22} \\
\hline
\tikz{\node at (0,0) {\ytableaushort{11,2}};} & \raisebox{0.2cm}1 & \raisebox{0.2cm}0 & \raisebox{0.2cm}0 \\
\tikz{\node at (0,0) {\ytableaushort{12,2}};} & \raisebox{0.2cm}0 & \raisebox{0.2cm}0 & \raisebox{0.2cm}{$0$} \\
\end{tabular}
\]

Using the correct maps that we obtained in this paper, we obtain the following matrices instead.
\[
\begin{tikzpicture}
\node at (0,0) {
\begin{tabular}{c|ccc}
$\mc{F}_{(2), (2,1)}(x_1)$ & \ytableaushort{11} & \ytableaushort{12} & \ytableaushort{22} \\
\hline
\tikz{\node at (0,0) {\ytableaushort{11,2}};} & \raisebox{0.2cm}0 & \raisebox{0.2cm}{$\frac 1 2$} & \raisebox{0.2cm}0 \\
\tikz{\node at (0,0) {\ytableaushort{12,2}};} & \raisebox{0.2cm}0 & \raisebox{0.2cm}0 & \raisebox{0.2cm}{$1$} \\
\end{tabular}
};
\node at (8,0) {
\begin{tabular}{c|ccc}
$\mc{F}_{(2), (2,1)}(x_2)$ & \ytableaushort{11} & \ytableaushort{12} & \ytableaushort{22} \\
\hline
\tikz{\node at (0,0) {\ytableaushort{11,2}};} & \raisebox{0.2cm}{$-1$} & \raisebox{0.2cm}0 & \raisebox{0.2cm}0 \\
\tikz{\node at (0,0) {\ytableaushort{12,2}};} & \raisebox{0.2cm}0 & \raisebox{0.2cm}{$-\frac 1 2$} & \raisebox{0.2cm}0 \\
\end{tabular}
};
\end{tikzpicture}
\]

\section*{Acknowledgments}
We thank Reuven Hodges for important discussions.
The authors were supported by the DFG grant IK 116/2-1.

	\bibliographystyle{alpha}
%	\bibliography{bibliography}

\end{document}